\documentclass[11pt]{amsart}
\usepackage{fullpage}
\usepackage[centertags]{amsmath}
\usepackage{amsfonts}
\usepackage{amssymb}
\usepackage{amsthm}

\newtheorem{theorem}{Theorem}[section]

\theoremstyle{definition}
\newtheorem{definition}[theorem]{Definition}

\theoremstyle{remark}

\newcommand{\V}{\boldsymbol{V}}
\newcommand{\s}{\mathcal{S}}
\newcommand{\e}{\mathrm{e}^}
\newcommand{\BL}{\mathcal{L}}
\newcommand{\lx}{\mathrm{ch}(G)}

\begin{document}

\title[List coloring and the Potts model]{A note on recognizing an old friend in a new place: list coloring and the zero-temperature Potts model}

\author[J.~Ellis-Monaghan]{Joanna A. Ellis-Monaghan}
\address{Department of Mathematics, Saint Michael's College, 1 Winooski Park, Colchester, VT 05439, USA.  }
\email{jellis-monaghan@smcvt.edu}
\thanks{The work of the first author was supported by the National Science Foundation (NSF) under grant number DMS-1001408.  This paper's contents are solely the responsibility of the authors and do not necessarily represent the official views of the NSF}

\author[I.~Moffatt]{Iain Moffatt}
\address{Department of Mathematics, Royal Holloway, University of London, Egham, Surrey, TW20 0EX, United Kingdom}
\email{iain.moffatt@rhul.ac.uk}
\thanks{Some of this work was undertaken while the authors were visiting the Erwin Schr\"odinger International Institute for Mathematical Physics (ESI) in Vienna. We would like to thank the ESI for their support and for providing a productive working environment}



\date{\today}

 \begin{abstract}
 Here we observe that list coloring in graph theory coincides with the zero-temperature antiferromagnetic Potts model with an external field. We give a list coloring polynomial that equals the partition function in this case. This is analogous to the well-known connection between the chromatic polynomial and the zero-temperature, zero-field, antiferromagnetic Potts model.  The subsequent cross fertilization yields immediate results for the Potts model and suggests new research directions in list coloring.
\end{abstract}

% email secr@esi.ac.at with cc director@esi.ac.at
% subject: ESI research documentation/ Combinatorics, Geometry and Physics_ATV_2014
% body: author/title/arxiv or journal ref/ Combinatorics, Geometry and Physics_ATV_2014

\maketitle

\section{Introduction}\label{s.intro}
One of the many  fruitful connections between combinatorics and physics arises from the recognition that, with constant interaction energy and in the absence of an external field, the classical Tutte polynomial of graph theory and the Potts model of statistical mechanics are the same object. (Surveys of this connection may be found in \cite{BE-MPS10,Bol98,WM00,Sok05,We93}.) A relevant subcase of this is that the chromatic polynomial corresponds to the zero-temperature limit of the antiferromagnetic Potts model.

 In this note we offer simply an observation, but one which we believe builds a new and important bridge between the fields of graph theory and statistical mechanics.  We find an unexpected connection between a heavily-studied area of graph theory and the zero-temperature Potts model \emph{with} an external field. This connection not only provides new opportunities for researchers to  further reap the rewards of cross-pollination between the fields, but it also provides a formalization for, and gives an established graph theoretical foundation to, work in this area undertaken from a physics perspective.  For example, it parallels the work of  Shrock and Xu in \cite{SX10a, SX10b}, where  they study a partition function that is also a specialization of the $\V$-polynomial, but in a different limit.

 List coloring is an important generalization of proper graph coloring. A  list coloring  of a graph  is a proper coloring of it using, at each vertex, only colors drawn from a list specified for that vertex.
 It has been extensively studied since its introduction by V. Vizing \cite{Viz76} and by P. Erdos, A. Rubin and H. Taylor \cite{ERT79}, and some beautiful list coloring results have been obtained over the years.  The Potts model is a central object of study in statistical mechanics, and many of its applications involve some form of an external field.  Of interest in this model is the antiferromagnetic zero-temperature limit.  In this note we define a  list-chromatic polynomial, present it as a specialization of the $\V$-polynomial of \cite{EMM11}, and show that it agrees with the partition function for the zero-temperature antiferromagnetic Potts model with an external field.
 This extends the classic result that the chromatic polynomial corresponds to the zero-temperature limit of the antiferromagnetic Potts model in a zero-field to the setting of non-trivial external fields.
Our observation has immediate ramifications in both areas of study.

List coloring is fundamentally different from classical proper coloring. For example, while the proof of the Four Color Theorem for plane graphs is extremely complicated, the proof of the corresponding Five Choosability Theorem for list coloring plane graphs is both short and elegant (see \cite{Th94}).  Furthermore, while the chromatic number gives a lower bound for the choosability number, the gap between these two numbers can be arbitrarily large, even when restricted to bipartite graphs. Consequently, when there is no external field, there will be a zero energy ground state provided that the number of spin states is at least as large as the chromatic number, but in contrast, even if there are considerably more spin states, external field contributions be may chosen to stop the system settling into a zero energy ground state.  We give an example of this phenomenon on an augmented 3D cubic lattice.

The correspondence between external fields and list coloring leads to a number of other consequences, such as  a unified model incorporating boundary conditions, noting that phase transition is independent of preference strengths, reduction to the classical case for uniform spin preferences, and computational complexity implications.  We conclude with a some suggestions  of a few new directions in list coloring research suggested by the Potts model application.

\section{A review of the $\V$-polynomial and the Potts model}
In this section we provide a brief review of  the $\V$-polynomial and its specialization to the Potts model with an external field. Full details can be found in \cite{EMM11}.

Let $G=(V,E)$ be a graph (we allow multiple edges and loops). We will generally assume that $V=\{v_1, \ldots , v_n\}$, and will often specify a vertex $v_i$ by its index $i$.
 Edges will be identified either by name such as $e$, or by their endpoints as $\{i, j\}$. Although graphs may have multiple edges,  we eschew the bulky formalism of indexing the multiple edges as there is no danger of confusion here.  Since we will make frequent use of functions whose domain is either $E$ or $V$, we often denote function values by subscripts, for example writing $F_a$ for $F(a)$.

We work with graphs that are both vertex and edge weighted.  Here a {\em vertex weighted graph} is a graph $G$, together with a  weight function $\omega$ mapping $V(G)$ into a  torsion-free commutative semigroup. The {\em weight} of the vertex $i$ is the value $\omega_i$. Similarly, an {\em edge weighted} graph $G$ has a function $\gamma$ from its edge set $E(G)$, to a set $\boldsymbol{\gamma}:=\{\gamma_e\}_{e\in E(G)}$, with  $\gamma:e\mapsto \gamma_e$, for $e\in E(G)$. We use $\boldsymbol{x}$ to denote an indexed set of variables.

If $G$ is a vertex and edge weighted graph, and $e$ is an edge of $G$, then $G-e$ is the graph obtained from $G$ by deleting the edge $e$, leaving the vertex weight function unchanged, and restricting the edge weight function to $E(G)-\{e\}$. If $e$ is any non-loop edge of $G$, then $G/e$ is the graph obtained from $G$ by contracting the edge $e$ and changing  the vertex weight function as follows: if $v_i$ and $v_j$ are the vertices incident to $e$, and $v$ is the vertex of $G/e$ created by the contraction, then $ \omega(v) =  \omega(v_i) +\omega(v_j) $. In this case again the edge weight function is simply restricted to $E(G)-\{e\}$. Loops are not contracted.

\begin{definition}\label{mvw}
Let $S$ be a torsion-free commutative semigroup, let $G$ be a graph equipped with vertex weights $\boldsymbol{\omega}:=\{ \omega_i \} \subseteq S$  and edge weights $\boldsymbol{\gamma}:= \{\gamma_e\}$, and let $\boldsymbol{x}:=\{x_k\}_{k\in S}$ be a set of commuting variables. Then the {\em $\V$-polynomial} of  $G$,
denoted $\V(G) = \V(G, \omega; \boldsymbol{x}, \boldsymbol{\gamma}) \in \mathbb{Z} [ \{\gamma_e\}_{e\in E(G)}, \{x_k\}_{k\in S}  ], $ is defined recursively by:
\begin{align}
 \label{dc1} \V(G) &= \V(G-e)+\gamma_e  \V(G/e)\text{, if $e$ is a non-loop edge of $G$;} \\
\label{dc2} \V(G)&=(\gamma_e+1)\V(G-e)\text{, if $e$ is a loop;} \\
\label{dc3} \V(E_n)&= \prod_{i=1}^{n} x_{\omega_i}\text{, if $E_n$ consists of $n$ isolated vertices of weights  $\omega_1, \ldots , \omega_n$.}
\end{align}
\end{definition}

 It was shown in \cite{EMM11} that the $\V$-polynomial is well-defined, that is, it is independent of the order in which the deletion-contraction relation is applied to the edges. It was also shown that it admits a  spanning subgraph expansion:
\begin{equation}\label{vspan} \V(G)=\sum_{A\subseteq E(G)}   x_{c_1}   x_{c_2}\cdots   x_{c_{k(A)}}   \prod_{e\in A} \gamma_e ,   \end{equation}
where $c_l$ is the sum of the weights of all of the vertices in the $l$-th connected component of the spanning subgraph $(V(G), A)$ of $G$.
Spanning tree and spanning forest expansions for the $\V$-polynomial were given in \cite{MM}.

\bigskip

 We use an expression for the Potts model with an external field in a form that can be specialized to many other common models.
A {\em state} of a graph $G=(V,E)$ is an assignment  $\sigma: V\rightarrow \{1,\ldots , q\} $, for $q\in \mathbb{N}$, where the value of $\sigma_i := \sigma (v_i)$ is  the \emph{spin} at the vertex $i$.   We let $\s(G)$ denote the set of states of $G$.

For a given graph state $\sigma$ of $G$, the  \emph{Hamiltonian of  the Potts model with variable edge interaction energy and variable field} is
\begin{equation}\label{e.ham2}
h(\sigma) =  -\sum_{ \{ i,j \} \in E(G) }  J_{i,j} \delta(\sigma_i, \sigma_j)  - \sum_{v_i\in V(G)}    \sum_{\alpha =1}^q M_{i,\alpha}  \delta(\alpha,  \sigma_i).
\end{equation}
Here each edge $e=\{i, j\}$ has an associated  interaction energy (or spin-spin coupling)  $J_e=J_{i,j}$. The  $M_{i,\sigma_i}\in \mathbb{C}$ are the external field contributions. For a vertex $v_i$ with spin $\sigma_i$, the external field contributes $M_{i,\sigma_i}$ to the Hamiltonian. We record the possible (scalar) values of the  external field contributions at a vertex $v_i$ formally as a vector $ \boldsymbol{M}_i :=(  M_{i,1}, M_{i,2}, \ldots , M_{i,q}  )\in \mathbb{C}^q$.  Accordingly, we may view our graphs here as having edge weights given by $\{J_e\}_{e\in E(G)}$ and vector-valued vertex-weights $\{\boldsymbol{M}_i\}_{i\in V(G)}$. The Hamiltonian of Equation~\eqref{e.ham2} specialises to many other common Hamiltonians from the literature (see Section~6 of \cite{EMM11} for details).

The {\em Potts model partition function} is
\begin{equation*}%\label{z.sum}
Z(G, \boldsymbol{J},\boldsymbol{M};q,T) := \sum_{\sigma \in \s (G)}   \e{-\beta h(\sigma)}.
\end{equation*}
Here $\beta = 1/(\kappa T)$, where $T$ is the temperature of the system,  $\kappa
$ is the Boltzmann constant, and  $\boldsymbol{J}$ and $\boldsymbol{M}$ refer to the interaction energies and external field contributions on $G$, respectively.

\bigskip

 The main result of \cite{EMM11}, given here as Theorem \ref{ZV2} below, is that the Potts model partition function with a variable external field and variable edge interaction energies is an evaluation of the $\V$-polynomial. This  result directly extends the seminal connection between the zero-field Potts model partition function and the classical Tutte polynomial to one that incorporates external fields. Moreover,
appropriate specialisations of the result give graph polynomial connections, and Fortuin-Kasteleyn-type representations for various standard, widely studied models such as preferred spin, set of preferred spins, random field Ising model, etc., as given in \cite{EMM11}.
\begin{theorem}[\cite{EMM11}]\label{ZV2} Let $G$ be a graph with external field contributions given by  $\boldsymbol{M}_i=(M_{i,1}, \ldots , M_{i,q})\in \mathbb{C}^q$.
Then
\[  Z(G, \boldsymbol{J},\boldsymbol{M};q,T)=\V\left( G,\omega ; \;\{X_{\boldsymbol{M}}\}_{\boldsymbol{M}\in \mathbb{C}^q}    ,\;   \{  \e{\beta J_{i,j}}-1   \}_{\{i,j\}\in E(G)}      \right),\]
where the vertex weights  are given by $\omega (v_i) =\boldsymbol{M}_i$
 and, for any $\boldsymbol{M} = (M_1, \ldots , M_q) \in \mathbb{C}^q$,
$ X_{\boldsymbol{M}}=  \sum_{\alpha=1}^q \e{\beta M_{\alpha}}$.
\end{theorem}

Since the $\V$-polynomial has a deletion-contraction reduction (Equations~\eqref{dc1}--\eqref{dc3}), this means that, like the classical case, the Potts model with an external field does too.  A further consequence is the following Fortuin-Kasteleyn-type representation for the Potts model  with variable external field and variable edge interaction:
\begin{equation}\label{fk}  Z(G, \boldsymbol{J},\boldsymbol{M};q,T)= \sum_{A\subseteq E(G)}  X_{\boldsymbol{M}_{C_1}} \cdots X_{\boldsymbol{M}_{C_{k(A)}}}  \prod_{e\in A}  (\e{\beta J_e}-1) ,
\end{equation}
where $\boldsymbol{M}_{C_l}$ is the sum of the weights, $\boldsymbol{M}_i$, of all of the vertices $v_i$ in the $l$-th connected component of the spanning subgraph $(V(G), A)$, and  $X_{\boldsymbol{M}}=      \sum_{\alpha=1}^q \e{\beta M_{\alpha}}$,  for  $\boldsymbol{M} = (M_1, \ldots , M_q) \in \mathbb{C}^q$.

\section{List coloring and the zero-temp antiferromagnetic Potts model}

\subsection{List coloring}

 List colorings are generalizations of proper graph colorings in which each vertex $v_i$ is assigned a set (called a {\em list}) $l_i$ of allowed colors, and in which the coloring may only assign colors from $l_i$ to $v_i$.  Formally, let $\BL$ be a set ($\BL$ is just the large set from which the vertex lists are drawn; typically $\BL= \mathbb{N}$ ).  Given a graph $G=(V,E)$ and a set of list $L=\{l_i \subseteq \BL \}_{i\in V}$,    we say $G$ is \emph{$L$-colorable} if there is a proper coloring of $G$ with the color at each vertex $v_i$ belonging to the list $l_i$.  Note that if each $l_i=\{1,2,\ldots, k\}$, then $L$-colorability is exactly $k$-colorability.
A graph $G$ is \emph{$k$-choosable} for a fixed integer $k$ if it is $L$-colorable for any set of lists $L$ with $|l_i|=k$ for all $i$. The \emph{choosability}, $\lx$, of $G$ is the smallest $k$ such that $G$ is $k$-choosable. Note that although we say `list' here to be consistent with the literature, the $l_i$'s are simply sets. There are no repeated elements, nor does the order of the elements matter. Also, while these lists may in general be infinite, for the current application it suffices to assume they are finite.

\subsection{The list-chromatic polynomial}
Recall that the chromatic polynomial $\chi(G; \lambda)$ is the graph polynomial whose evaluation  $\chi(G; k)$ is the number of proper $k$-colorings of $G$, for each $k\in \mathbb{N}$.
Here we introduce an analogue of the chromatic polynomial for list colorings.
We will show that, just as the chromatic polynomial is a specialization of the Tutte polynomial (see, \cite{Bol98} for details on the chromatic and Tutte polynomials), the list-chromatic polynomial is a specialization of the $\V$-polynomial.

\begin{definition} Let $G=(V,E)$ be a graph with lists $L=\{l_i\}_{i\in V}$ drawn from some set $\BL$.  Let S be the semigroup $2^L$  under intersection. Assign each edge $e$ of $G$ the weight $\gamma_e=-1$.
Then the \emph{list-chromatic polynomial} is defined by
\begin{equation*}
P(G,L) := \V (G, L; \boldsymbol{x}, \boldsymbol{-1} ).
\end{equation*}
\end{definition}

Since $P(G,L)$ is a specialization of $\V$, it inherits a  linear recursion relation:
\begin{align}
   \label{nonloop}P(G,L) &= P(G-e,L')-P(G/e,L'') \text{, when $e$ is not a loop;} \\
   \label{loop} P(G,L)& =0\text{, when $e$ is a loop;} \\
   \label{noedge} P(G,L)&= \prod_{i=1}^{n} x_{l_i}\text{,  when $G$ has $n$ vertices and no edges.}
\end{align}
Here $L'=L$ since $V(G)=V(G-e)$, and the lists in $L''$ are the same as in $L$ except at the endpoints of $e$, where the list at the new merged vertex in $G/e$ is given by the intersection of the two lists at the endpoints of $e$ in $G$.

 $P(G,L)$ also inherits a spanning tree expansion from $\V$:
\begin{equation}\label{listchrom2}
 P(G,L)=\sum_{A\subseteq E(G)}  (-1)^{|A|} x_{c_1}   x_{c_2}\cdots   x_{c_{k(A)}}  ,
\end{equation}
where $c_l$ is the intersection of the lists of all of the vertices in the $l$-th connected component of the spanning subgraph $(V(G), A)$. In addition, spanning tree and spanning forest expansions for $P(G,L)$ follow from results in \cite{MM}.

\begin{theorem} Let $G=(V,E)$ be a graph with lists $L=\{l_i\subseteq \mathcal{L}\}_{i\in V}$. The list-chromatic polynomial counts list colorings of $G$  in that $P(G,L)|_{ \{ x_{l}=|l| \} }$   is the number of ways to list color $G$ using the lists $L$.  Here, for $l\subseteq \mathcal{L}$, $\{x_{l}=|l|\}$ means to substitute $|l|$ for each indeterminate $x_l$.
\end{theorem}
 \begin{proof}  The proof is a routine induction on the number of non-loop edges of $G$.  If $G$ has no non-loop edges, the result is immediate from  Equations \eqref{loop} and \eqref{noedge}.  Now let $e$ be a non-loop edge of $G$, and consider $G-e$ with lists given by $L'$.  The number of ways to list color $G-e$ is the number of ways when the endpoints of $e$ receive different colors, plus the number of ways when the endpoints of $e$ receive the same color.  This is just the number of ways to list color $G$ with lists from $L$ plus the number of ways to list color $G/e$ with lists from $L''$, from which the result follows.
\end{proof}

Note that, as would be expected, $P(G,L)$ specializes to the usual chromatic polynomial  when all the lists are the same, i.e., when the list coloring reduces to the usual graph coloring.  In this case, only one variable $x_l$ appears in $P(G,L)$, and the polynomial coincides with the chromatic polynomial $\chi(G;x_l)$ of $G$.

\subsection{The zero-temperature antiferromagnetic Potts model with an external field coincides with list coloring}
 It is well known that the chromatic polynomial corresponds to the zero-temperature limit of the antiferromagnetic zero-field $q$-state Potts model (i.e., when each $M_{i,\alpha}=0$ in Equation~\eqref{e.ham2}):
\[  \lim_{T \rightarrow 0} Z(G, \boldsymbol{J},\boldsymbol{0};q,T)  =  \chi(G;q).  \]
 We will now prove the main observation of this note: that  in the presence of an external field, the zero-temperature limit of the Potts model gives list colorings.

\begin{theorem}\label{obs}
The list-chromatic polynomial corresponds to the zero-temperature limit of the antiferromagnetic  Potts model (with a possibly non-zero external field). That is, if $G=(V,E)$ is a graph then for the antiferromagnetic model in which all external field contributions are non-positive, we have that
\[  \lim_{T \rightarrow 0} Z(G, \boldsymbol{J},\boldsymbol{M};q,T) =  P(G,L),   \]
where each list $l_i$ at $v_i$ is given by the positions of the zero terms of the external field contributions $\boldsymbol{M}_i$.
\end{theorem}
\begin{proof}
Let $G=(V,E)$ be a graph. For the antiferromagnetic model, $J_e <0$ for all edges, and we further assume that all field vectors have non-positive entries.

For the zero-temperature limit, as $T \rightarrow 0$ we have that   $\beta \rightarrow \infty$, so $\exp(-\beta h(\sigma))=0$ unless $h(\sigma) =0$, in which case it is 1.  Thus, in the limit, $Z(G, \boldsymbol{J},\boldsymbol{M};q,T)$ counts states $\sigma$ in which $h(\sigma)=0$.
Since $J_e < 0$  and $M_{i,\alpha} \leq 0$ for all $i$ and $\alpha$, both  $\sum\limits_{ \{ i,j \} \in E }  J_{i,j} \delta(\sigma_i, \sigma_j)$  and  $ \sum\limits_{v_i\in V}    \sum\limits_{\alpha =1}^q M_{i,\alpha}  \delta(\alpha,  \sigma_i)$ must then  be zero.  The former sum is zero if and only if $\sigma$ is a proper coloring.    For the latter sum, note that if we assign a list $l_i$ to a vertex $i$ by $r \in l_i \iff M_{i,r}=0$, then $ \sum\limits_{v_i\in V}    \sum\limits_{\alpha =1}^q M_{i,\alpha}  \delta(\alpha,  \sigma_i) =0 \iff \alpha \in l_i$.
Thus, $h(\sigma)=0 \iff \sigma$ gives a list coloring using the lists $l_i$.
\end{proof}

The negative entries in the external field contributions may be thought of as encoding non-prefered spins as in \cite{SX10a, SX10b}, but here the entries are not restricted to an interval as in that previous work.  Since, as noted in Section~\ref{ramifications}, the values of the entries do not matter, the zero entries in the external field contributions may be thought of as preferred spins.

Note that in the absence of an external field (when all $M_{i,\alpha}=0$) we have $l_i = \{1, \ldots , q\}$ for all $i$. Then $Z(G, \boldsymbol{J},\boldsymbol{0};q,T)$ counts proper colorings, which is just the classical connection between the chromatic polynomial and the zero-temperature Potts model.

\subsection{Phase transitions}
  Since the fundamental thermodynamic functions involve taking logarithms of the partition functions, phase transitions occur when the infinite volume limit of the partition function is zero.  In particular, we consider the  ground state entropy (per vertex) of the Potts
antiferromagnetic model in the infinite volume limit:
 \[\kappa \lim_{n \rightarrow \infty}  \lim_{T \rightarrow 0} \frac{1}{n} \ln (Z(G_n, \boldsymbol{J},\boldsymbol{M};q,T)) =\kappa \lim_{n \rightarrow \infty} \frac{1}{n} \ln (P(G_n;L_n)).
 \]
 Here $n \rightarrow \infty$ means that there is an infinite family of graphs (typically lattices) that grow in size, with some modest constraints, such as there be natural inclusions $G_1 \subseteq \ldots \subseteq G_n \subseteq \ldots$, and that the external field contributions or lists respect these inclusions so that the external field or list on a vertex in $G_i$ remains on that vertex under its inclusion in all $G_j$ for $i<j$.

\bigskip

We pause here for some brief historical notes. On the combinatorics side, D. Forge and  T. Zaslavsky, in  \cite{FZ13}, studied a polynomial of gain graphs, and their work includes a list-chromatic polynomial for gain graphs.  Their polynomial, developed independently, is essentially an evaluation of the $\V$-polynomial of \cite{EMM11}, specialized to gain graphs, and their list-chromatic polynomial for gain graphs is similar in flavor, if not in detail, to the list-chromatic polynomial given here.

On the statistical physics side,  R. Shrock and Y. Xu considered the {\em weighted-set chromatic polynomial},  $Ph(G)$, in \cite{SX10a,SX10b}, and note but do not pursue a possible list-coloring connection. Like the list-coloring polynomial, the polynomial $Ph(G)$ is an extension of the chromatic polynomial that arises from the Potts model partition function in an external field. The polynomial $Ph(G)$ can  be recovered from the $\V$-polynomial by removing the $\e{\beta J_e}$ from the product in \eqref{fk}, and taking each $\boldsymbol{M}=(H,\ldots, H, 0,\ldots, 0)$ where the number of $H$'s is $s$. However, the polynomial  $Ph$ is not the list-coloring polynomial. $Ph$ is obtained from the Potts model by taking the limit $K=\beta J\rightarrow \infty$, rather than $\beta \rightarrow \infty$ as we do here to obtain list coloring (in $Ph$, taking $\beta \rightarrow \infty$ corresponds to $w\rightarrow 1$ which gives the chromatic polynomial). It would be valuable to have combinatorial framework for studying $Ph$.

\section{Ramifications and perspectives for future work} \label{ramifications}
We provide  a few examples of ramifications of Theorem~\ref{obs}  to illustrate that meaningful results should emerge by mining the relevant statistical mechanics and graph theory literature, and correlating results across the fields.

$\bullet$ \emph{Contrast to the zero-field case for minimum energy states.} In the zero-field case, understanding minimum energy states  is straightforward:  a system can achieve a zero energy state if and only if the number of spin states is at least as large as the chromatic number, $\gamma(G)$, of its graph $G$.  This is not the case in the presence of an external field.  Here, if $q<\gamma(G)$ then again the system cannot achieve a zero energy state.  However, if $q \geq \gamma(G)$, then whether or not the system is able to achieve a zero energy state is highly dependent on the exact form of the external field contributions.  If all lists have size greater than $\lx$, then the system will be able to achieve a zero energy state. In general, since the external field contributions determine the lists on the vertices, the question  becomes whether or not the external field contributions give rise to lists from which $G$ is list colorable.

For example,   let $q=3$ and consider the cubic lattice augmented by adding the three internal diagonal edges to each unit cube, so that each cube is a copy of $K_{4,4}$.  Assign the following external field contributions according to the coordinates of the  vertices: $(-1,0,0)$ if $x+y \equiv 0 \mod 3$, then $(0,-1,0)$ if $x+y \equiv 1 \mod 3$, and $(0,0, -1)$ if $x+y \equiv 2 \mod 3$, irrespective of the $z$-coordinate.  This system does not have a zero energy state since the lattice cannot be colored from the resulting lists.

On the other hand, Thomassen \cite{Th94} showed that planar graphs are 5-choosable.  This means that in a planar system with  $q$ spin states,  as long as there are at least 5 zero entries in the external field contribution at each vertex (so the lists all have size at least 5), then the system will have zero energy states.

$\bullet$ \emph{Independence of preference strength.}  In the conversion from the Potts model to list coloring, the values of the entries in the external field contributions do not matter, just their positions.  Thus, any phase transition or other zero-temperature phenomenon captured by the partition function is independent of the strengths of the spin preferences and only depends on which spins are preferred.

$\bullet$ \emph{Boundary conditions.} Many models, both with and without external fields, assume boundary conditions in the form of fixed spins on specified vertices.  The formalism here gives a unified approach to zero-temperature limits of antiferomagnetic boundary condition models.  A boundary vertex $i$ with a fixed spin of, say, $b$ is simply assigned a field contribution with $M_{i,\alpha}$ equal to 0 if $\alpha =b$, and -1 otherwise.  This forces its spin to be $b$ in every state counted in $\lim\limits_{T \rightarrow 0} Z(G, \boldsymbol{J},\boldsymbol{M};q,T)$.
In particular, there is now a deletion-contraction reduction and Fortuin-Kasteleyn-type  expansion for models with boundary conditions in the zero-temperature limit.

Focussing on the combinatorics, precoloring is a form of graph coloring in which colors are specified on some vertices of the graph, and then the rest of the vertices are colored so as to achieve a proper coloring.   Precoloring thus corresponds to the zero-temperature limit of a system with boundary conditions.

 In \cite{JS},  J. Jacobsen and H. Saleur studied a variant of the chromatic polynomial, called the  {\em boundary chromatic polynomial}, $P_G(q,q_s)$ that arose from the Potts model partition function with boundary conditions. They considered $q$-colorings of graphs embedded in the annulus in which vertices on the boundary can only be coloured with a subset of $q_s$ colors. Their polynomial $P_G(q,q_s)$ is easily recovered as an evaluation of $P(G;L)$  by giving the non-boundary vertices  lists $ \{1, \ldots ,q\}$ and the boundary vertices a list consisting of the allowed $q_s$ colours, and then applying  Equation~\eqref{listchrom2}. Note that applying Equation~\eqref{vspan} instead gives the polynomial $Z_G(q,q_s; \boldsymbol{v})$ from \cite{JS}.

$\bullet$ \emph{Uniform preferences.}   We say that  external field contributions are \emph{$s$-uniform} if they have $s$ zero entries and are the same at each vertex, as is the case in many common models.   In this case, the corresponding lists at each vertex in the zero-temperature limit all have size $s$ and are the same.  Thus, at zero-temperature, the partition function reduces to the chromatic polynomial but evaluated at  $s$ not $q$.  This gives a succinct mathematical proof that the zero-temperature phase transitions for preferred spin models will coincide with those of  a zero-field model with fewer spin states. For example, if $q=3$, then  a system with two preferred spins (so one non-preferred spin, and the external field contributions have two zero entries, so the lists have size two),  will have the same zero-temperature phase transition as it does for the classical zero-field Ising model.

$\bullet$\emph{Computational complexity.}  In the zero-field case there has been considerable interplay of computational complexity results from the fields of graph theory, and statistical mechanics  (see \cite{We93}, for example).  This appears to be the case here as well,  as indicated by the following sample results.  We begin by noting that the number of minimum energy states is an important statistic.  In the case that the minimum energy is zero, this number is given by evaluating the list coloring polynomial.  As demonstrated by a number of results from the list coloring literature, this is often NP-hard in general. However, there are  tractable special cases, particularly for graphs of bounded treewidth.  As a representitive example, in \cite{JS92} we find that  list coloring can be done in $O(n^{t+2})$-time where $n$ is the size of the graph and $t$ is its treewidth.    Thus, the number of zero energy states may be computed in polynomial time for systems with bounded treewidth.  A second example, from  \cite{Fetal11}, is that for graphs of constant treewidth, given a constant $r$, there is a linear time algorithm to tell if $\lx \leq r$ when all the lists on $G$ have size at least $r$.  This means that for such graphs, there is a linear time algorithm that will determine that there are zero energy states given $r$ as the minimum number of zeros appearing in any of the external field contributions.

\medskip

Recognizing this new application of list coloring to statistical physics now leads to  new directions in list coloring.  We summarize some areas of interest in statistical mechanics  restated as list coloring problems.  These problems often shift the emphasis of list coloring problems from choosability in general, to specific classes of graphs (especially lattices) and to specific lists.

$\bullet$ \emph{Families of graphs.} In statistical mechanics, the behaviour  of partition functions of families of graphs that grow in some controlled way, such as square, triangular, hexagonal, or cubic lattices, $C_n \times C_n$, ladders, etc. are often considered. In terms of list coloring, these considerations place an emphasis on list coloring results for these families of graphs: when may or may not they be list colored and from what types of lists.  In the questions below, again these would be some relevant classes of graphs to consider.

$\bullet$ \emph{Phase transitions.} A fundamental question in statistical mechanics is whether there is a phase transition at zero-temperature.  This manifests as a failure of analyticity in the infinite volume limit.  In the zero-field case, this problem has been approached by considering accumulation points, and by clearing regions of the complex plane to show that there can be no real accumulation point in some particular interval.  Here, however, the problem becomes multi-dimensional as the parameters are not just a single integer $q$, but all integers from 0 to $q$, representing the possible list sizes.  Thus, an analogous approach would involve clearing regions of $\mathbb{C}^q$ for the list coloring polynomial for specific classes of graphs.

$\bullet$\emph{Counting.} The number of minimum energy states of a system is a relevant statistic.  Assuming the minimum energy is zero, this translates to asking how many ways a graph $G$ may be colored from a given set of lists.  As noted above, this is given by evaluating the list-chromatic polynomial, and is known to be NP-hard in many cases.  However, even bounds for relevant classes of graphs or particular forms of lists would be interesting.

\bibliographystyle{plain}

\end{document}